\documentclass[a4paper,11pt]{amsart}

\usepackage{a4wide}
\usepackage{amsmath, amsfonts, amscd, amssymb, amsthm}
\usepackage{tikz}
\usepackage[UKenglish]{babel}
\usetikzlibrary{matrix,arrows,calc}
\usepackage[all]{xy}
\usepackage{xspace}
\usepackage[T1]{fontenc}

\newcommand{\thedocumentname}{Classification of boundary Lefschetz Fibrations}
\newcommand{\theauthor}{Stefan Behrens, Gil R.\ Cavalcanti and Ralph L.\ Klaasse}

\usepackage[pdftitle={\thedocumentname},pdfauthor={\theauthor},bookmarks,urlcolor=black,colorlinks=true,linkcolor=black,citecolor=black]{hyperref}
\usepackage{aliascnt}

\addto\extrasUKenglish{%

}

\numberwithin{equation}{section}							
\theoremstyle{plain}
\newtheorem{theorem}{Theorem}[section]

\newtheorem*{theorem*}{Theorem}
\newaliascnt{lemma}{theorem}
\newtheorem{lemma}[lemma]{Lemma}
\aliascntresetthe{lemma}

\newaliascnt{proposition}{theorem}
\newtheorem{proposition}[proposition]{Proposition}
\aliascntresetthe{proposition}

\newaliascnt{corollary}{theorem}
\newtheorem{corollary}[corollary]{Corollary}
\aliascntresetthe{corollary}

\theoremstyle{definition}
\newaliascnt{definition}{theorem}
\newtheorem{definition}[definition]{Definition}
\aliascntresetthe{definition}

\newaliascnt{remark}{theorem}
\newtheorem{remark}[remark]{Remark}
\aliascntresetthe{remark}

\newaliascnt{exa}{theorem}
\newtheorem{exa}[exa]{Example}
\aliascntresetthe{exa}

\newcommand{\Z}{\mathbb{Z}}
\newcommand{\R}{\mathbb{R}}
\newcommand{\C}{\mathbb{C}}

\newcommand{\bi}{\begin{itemize}}
	\newcommand{\ei}{\end{itemize}}
\newcommand{\be}{\begin{equation*}}
	\newcommand{\ee}{\end{equation*}}
\newcommand{\bp}[1][]{\begin{proof}[Proof#1.]}				
	\newcommand{\ep}{\end{proof}}
\newcommand{\wt}{\widetilde}

\newcommand{\mc}{\mathcal}
\newcommand{\gcs}{generalized complex structure\xspace}

\newcommand{\lf}{Lefschetz fibration\xspace}

\newcommand{\blf}{boundary \lf}

\newcommand{\bff}{boundary fibration\xspace}

\newcommand{\lfs}{Lefschetz fibrations\xspace}
\newcommand{\blfs}{boundary \lfs}
\newcommand{\del}{\partial}

\newcommand{\ra}{\rightarrow}
\newcommand{\inv}{^{-1}}
\newcommand{\MCG}{\mathcal{M}}
\newcommand{\scp}[1]{\langle{#1}\rangle}
\newcommand{\CPbar}{\overline{\C P^2}}
\newcommand{\CP}{{\C P}^2}

\begin{document}

\author{Stefan Behrens \ \ \ \ \ \ Gil R.\ Cavalcanti \ \ \ \ \ \ Ralph L.\ Klaasse}
\address{Stefan Behrens. \rm{Department of Mathematics, Utrecht University, 3508 TA Utrecht, The Netherlands}}
\email{s.behrens@uu.nl}

\address{Gil R.\ Cavalcanti. \rm{Department of Mathematics, Utrecht University, 3508 TA Utrecht, The Netherlands}}
\email{g.r.cavalcanti@uu.nl}

\address{Ralph L.\ Klaasse. \rm{Department of Mathematics, Utrecht University, 3508 TA Utrecht, The Netherlands}}
\email{r.l.klaasse@uu.nl}

\thanks{S.B.\ was supported by VICI grant number 639.033.312, and G.C.\ and R.K.\ were supported by VIDI grant number 639.032.221 from NWO, the Netherlands Organisation for Scientific Research.}

\dedicatory{Dedicated to Professor Nigel Hitchin on the occasion
     of his $70$th birthday}

\begin{abstract}
We show that a four-manifold admits a boundary Lefschetz fibration over the disc if and only if it is diffeomorphic to $S^1 \times S^3\# n \overline{\C P^2}$, $\# m\C P^2 \#n\overline{\C P^2}$ or $\# m (S^2 \times S^2)$. Given the relation between \blfs and stable generalized complex structures, we conclude that the manifolds $S^1 \times S^3\# n \overline{\C P^2}$, $\#(2m+1)\C P^2 \#n\overline{\C P^2}$ and $\# (2m+1) S^2 \times S^2$ admit stable structures whose type change locus has a single component and are the only four-manifolds whose stable structure arise from \blfs over the disc.
\end{abstract}

\title{Classification of boundary Lefschetz fibrations over the disc}
\maketitle

\vspace{-2em}
\tableofcontents

%
\section{Introduction}

Generalized complex structures, introduced by Hitchin \cite{MR2013140} and Gualtieri \cite{MR2811595} in~2003, are geometric structures which generalize simultaneously complex and symplectic structures while at the same time providing the mathematical background for string theory. One feature of generalized complex geometry is that the structure is not homogenous. In fact, a single connected generalized complex manifold may have complex and symplectic points. This lack of homogeneity is governed by the {\it type} of the structure, an integer-valued upper semicontinuous function on the given manifold which tells ``how many complex directions'' the structure has at the given point. In particular, on a $2n$-dimensional manifold, points of type $0$ are symplectic points, while points of type $n$ are complex.

Among all type-changing generalized complex structures, one kind seems to deserve special attention: stable generalized complex structures. These are the structures whose canonical section of the anticanonical bundle vanishes transversally along a codimension-two submanifold, $\mathcal{D}$, endowing it with the structure of an elliptic divisor in the language of \cite{Cavalcanti:2015uua}. Consequently, the type of such a structure is $0$ on $X \setminus \mc{D}$, while on $\mc{D}$ it is equal to two. Many examples of stable generalized complex structures were produced in dimension four \cite{MR2574746,MR3531974,MR2958956,MR3177992} and a careful study was carried out in \cite{Cavalcanti:2015uua}. One of the outcomes of that study was that it related stable generalized complex structures to symplectic structures on a certain Lie algebroid. 

\begin{theorem*}[{\cite[Theorem 3.7]{Cavalcanti:2015uua}}]
Let $\mathcal{D}$ be a co-orientable elliptic divisor on $X$. Then there is a correspondence between gauge equivalence classes of stable generalized complex structures on $X$ which induce the divisor $\mathcal{D}$, and zero-residue symplectic structures on $(X,\mathcal{D})$.
\end{theorem*}
 
This results paves the way for the use of symplectic techniques to study stable structures. One result that exemplifies that use is the following.

\begin{theorem*}[{\cite[Theorem 7.1]{2017arXiv170303798C}}]\label{theo:blf to gcs}
Let $X^4$ be a closed connected and orientable four-manifold and let $\Sigma$ be a compact connected and orientable two-manifold with boundary $Z = \del \Sigma$. Let $f\colon X^4 \to \Sigma^2$ be a boundary Lefschetz fibration for which $\mathcal{D} = f^{-1}(\del \Sigma)$ is a co-orientable submanifold of $X$, and with $0 \neq [f^{-1}(p)] \in H_2(X \setminus \mathcal{D};\R) $, where $p \in \Sigma$ is a regular value of $f$. Then $X$ admits a stable \gcs\ whose degeneracy locus is $\mathcal{D}$.
\end{theorem*}

This result is reminiscent of Gompf's original one \cite{GompfStipsicz}, showing that Lefschetz fibrations give rise to symplectic structures. It is also is similar in content to a number of other results relating structures which are close to being symplectic to maps which are close to being Lefschetz fibrations. These include the relations between near-symplectic structures and broken Lefschetz fibrations \cite{MR2140998}, and between folded symplectic structures and real log-symplectic structures and achiral Lefschetz fibrations \cite{MR2253445,2016arXiv160600156C,TOPO:TOPO12000}.

The upshot of these results is that they at the same time furnish (at least theoretically) a large number of examples of manifolds admitting the desired geometric structure, and provide us with a better grip on those structures. With this in mind, our aim here is to classify all four-manifolds which admit boundary Lefschetz fibrations over the disc. Our main result is the following (\autoref{theo:blf over D2}).

\begin{theorem*}
Let $f\colon X^4\to D^2$ be a relatively minimal boundary Lefschetz fibration and $\mathcal{D} = f^{-1}(\del D^2)$. Then $X$ is diffeomorphic to one of the following manifolds:
\begin{enumerate}
\item $S^1 \times S^3$;
\item $\#m (S^2 \times S^2)$, including $S^4$ for $m=0$;
\item $\#m \C P^2 \# n \overline{\C P^2}$ with $m> n \geq 0$.
\end{enumerate}
In all cases the generic fibre is nontrivial in $H_2(X\setminus \mathcal{D};\R)$. In case \emph{(1)}, $\mathcal{D}$ is co-orientable, while in cases \emph{(2)} and \emph{(3)}, $\mathcal{D}$ is co-orientable if and only if $m$ is odd.
\end{theorem*}

We use essentially the same methods that were used by Behrens~\cite{BehrensFoldsCusps} and Hayano~\cite{MR2801419,MR3251824}. 
We translate the problem into combinatorics in the mapping class group of the torus, and then translate combinatorial results back into geometry using handle decompositions and Kirby calculus. Hayano's work turns out to be particularly relevant.
In his classification of so-called genus-one simplified broken Lefschetz fibrations he was led to study monodromy factorizations of Lefschetz fibrations over the disc whose monodromy around the boundary is a signed power of a Dehn twist. It turns out that the same problem appears for \blfs.

\subsection*{Organization of the paper}
This paper is organised as follows. In \autoref{sec:bfs} we introduce the notions of boundary fibrations as well as boundary Lefschetz fibrations and summarise their basic properties. In \autoref{sec:blfs over D2} we start studying the easier question of classifying oriented boundary fibrations over $D^2$, then we move on to prove the main theorem. The proof uses a careful study of genus-one Lefschetz fibrations over the disc which allows us to use an induction argument on the number of singular fibres to achieve our goal.

\section{Boundary Lefschetz fibrations}\label{sec:bfs}

In view of our interest in stable generalized complex structures and the results mentioned in the Introduction, the basic object with which we will be dealing in this paper are boundary (Lefschetz) fibrations. In this section we review the relevant definitions and basic results regarding them. We will use the following language. A {\it pair} $(X,\mc{D})$ consists of a manifold $X$ and a submanifold $\mc{D} \subseteq X$. A {\it map of pairs} $f\colon (X,\mc{D}) \to (\Sigma,Z)$ is a map $f\colon X \to \Sigma$ for which $f(\mc{D}) \subseteq Z$. A {\it strong map of pairs} is a map of pairs $f\colon (X,\mc{D}) \to (\Sigma,Z)$ for which $f^{-1}(Z) = \mc{D}$.
\begin{definition}
Let $f\colon (X^{2n},\mathcal{D}^{2n-2}) \to (\Sigma^2,Z^1)$ be a strong map of pairs which is proper and for which $\mathcal{D}$ and $Z$ are compact.
\begin{itemize}
\item The map $f$ is a {\it boundary map} if the normal Hessian of $f$ along $\mathcal{D}$ is nondegenerate;
\item The map $f$ is a {\it boundary fibration} if it is a boundary map and the following two maps are submersions:
\begin{itemize}
\item[a)] $f|_{X\setminus \mathcal{D}}\colon X\setminus \mathcal{D} \to \Sigma\setminus Z$, and 
\item[b)]$f\colon \mathcal{D} \to Z$.
\end{itemize}
The condition that $f$ is a boundary fibration (in a neighbourhood of $\mathcal{D}$) is equivalent to the condition that for every $x \in \mathcal{D}$, there are coordinates $(x_1,\dots ,x_{2n})$ centred at $x$ and $(y_1,y_2)$ centred at $f(x)$ such that $f$ takes the form
\begin{equation}\label{eq:boundary singularity}
f(x_1,\dots ,x_{2n}) = (x_1^2 + x_2^2,x_3),
\end{equation}
where $\mathcal{D}$ corresponds to the locus $\{x_1=x_2=0\}$ and $Z$ to the locus $\{y_1=0\}$;
\item The map $f$ is a {\it boundary Lefschetz fibration} if $X$ and $\Sigma$ are oriented, $f$ is a boundary fibration from a neighbourhood of $\mathcal {D}$ to a a neighbouhood of $Z$ and $f|_{X^{2n}\setminus \mathcal{D}}\colon X\setminus \mathcal{D} \to \Sigma\setminus Z$ is a proper Lefschetz fibration, that is, for each critical point $x\in X\setminus \mathcal{D}$ and corresponding singular value $y \in \Sigma\setminus Z$, there are complex coordinates centred at $x$ and $y$ compatible with the orientations for which $f$ acquires the form
\begin{equation}\label{eq:Lefschetz singularity}
f(z_1,\dots, z_n) = z_1^2 + \cdots + z_n^2.
\end{equation}
\end{itemize}
\end{definition}


\begin{exa}[$S^1 \times S^3$]\label{ex:S1xS3}
In this example we provide $X = S^1 \times S^3$ with the structure of a \bff over the disc, as described in \cite[Example 8.3]{2017arXiv170303798C}. The map $f\colon S^1 \times S^3 \to D^2$ is a composition of maps, namely
	\begin{equation*}
	S^1 \times S^3 \to S^3 \to D^2,
	\end{equation*}
where the first map is projection onto the second factor and the last is the projection from $\C^2$ to $\C$, $(z_1,z_2) \mapsto z_1$, restricted to the sphere.
In \autoref{T:no Lefschetz} we will see that this is, in fact, the only example of a boundary fibration over~$D^2$.
\end{exa}

A few relevant facts about \blfs\ were established in \cite{2017arXiv170303798C}. Beyond the local normal form \eqref{eq:boundary singularity} for the map $f$ around points in $\mathcal{D}$ there is also a semi-global form for $f$ in a neighbourhood of $\mathcal{D}$:

\begin{theorem}[{\cite[Proposition 5.15]{2017arXiv170303798C}}]\label{theo:semilocal form}
Let $f\colon (X^{2n},\mathcal{D}^{2n-2}) \to (\Sigma^2,Z^1)$ be a boundary map which is a boundary fibration on neighbourhoods of $\mathcal{D}$ and $Z$ and for which $Z$ is co-orientable. Then there are
\begin{itemize}
\item neighbourhoods $U$ of $\mathcal{D}$ and $V$ of $Z$ and diffeomorphisms between these sets and neighbourhoods of the zero sections of the corresponding normal bundles, $\Phi_{\mathcal{D}}\colon U \to N_{\mathcal{D}}$ and $\Phi_Z\colon V \to \R \times Z$, and
\item a bundle metric $g$ on $N_{\mathcal{D}}$,
\end{itemize}
such that the following diagram commutes, where $\pi\colon N_{\mathcal{D}}\to \mathcal{D}$ is the bundle projection:
$$\xymatrix@R=18pt@C=12pt{
U\ar[rrrr]^{f}\ar[d]^{\Phi_{\mathcal{D}}} &~&~&~& V\ar[d]^{\Phi_Z}\\
N_{\mathcal{D}}\ar[rrrr]^{(\|\cdot\|^2_g,\, f|_{\mathcal{D}}\circ \pi)}&~&~&~& \R \times Z
}$$
\end{theorem}

The most obvious consequence of this theorem is that in the description above, the image of $f$ lies on one side of $Z$, namely in $\R_+\times Z$. At this stage this is a local statement, but if $Z$ is separating (i.e., represents the trivial homology class) this becomes a global statement: the image of $f$ lies in closure of one component of $\Sigma \setminus Z$ and hence we can equally deal with $f$ as a map between $X$ and a manifold with boundary, $\Sigma$, whose boundary is $Z$. In this paper we will be concerned with the case when $\Sigma$ is the two-dimensional disc.

\begin{corollary}\label{cor:fibres are tori}
Let $f\colon (X^4,\mathcal{D}^2)\to (\Sigma^2,Z^1)$ be a boundary fibration with connected fibres and for which $Z$ is co-orientable and $X$ is connected and orientable. Then the generic fibre of $f$ is a torus.
\end{corollary}
\begin{proof}
From \autoref{theo:semilocal form} we see that the level set $f^{-1}\circ \Phi_Z^{-1}(\varepsilon,y)$ with $\varepsilon >0$ is a surface which fibres over the level set of $f^{-1}\circ \Phi_Z^{-1}(0,y)$, which is a circle, hence $f^{-1}\circ \Phi_Z^{-1}(\varepsilon,y)$ must be a torus or a Klein bottle. If $X$ is orientable, $N_{\mathcal{D}} \setminus \mathcal{D}$ is also orientable and due to \autoref{theo:semilocal form}, $\Phi_Z\circ f\circ \Phi_{\mathcal{D}}^{-1}\colon U \subset N_{\mathcal{D}} \setminus \mathcal{D}\to \R \times Z\setminus\{0\}\times Z$ is a fibration, where $U$ is a neighbourhood of $\mathcal{D}$, hence the fibres must be orientable.
\end{proof}

\begin{remark}
In the case when $X$ is connected, $\Sigma$ is a surface with boundary $Z = \del \Sigma$, and $f\colon X \to \Sigma$ is surjective, we can lift $f$ to a cover of $\Sigma$ so that the fibres of the boundary Lefschetz fibration become connected. That is, this particular hypothesis  is not really a restriction on the fibration (see \cite[Proposition 5.23]{2017arXiv170303798C}). In what follows we will always assume this is the case.
\end{remark}

\begin{remark} As shown in \cite[Proposition 6.8]{2017arXiv170303798C}, a boundary Lefschetz fibration $f\colon (X^4, \mc{D}^2) \to (D^2, \del D^2)$ satisfies $\chi(X) = \mu$, where $\mu$ is the number of Lefschetz singular fibres.
\end{remark}

\subsection{Vanishing cycles and monodromy}
Lefschetz fibrations on four-manifolds can be described combinatorially in terms of their monodromy representations and vanishing cycles.
We now extend this approach to \blfs.
For simplicity, we focus on fibrations over the disc and assume that they are injective on their Lefschetz singularities.
The latter condition can always be achieved by a small perturbation and the generalization to general base surfaces is exactly as in the Lefschetz case.

\begin{definition}[Hurwitz systems]\label{D:Hurwitz systems}
Let $f\colon (X^4,\mathcal{D}^2)\ra (D^2,\del D^2)$ be a \blf with~$\ell$ Lefschetz singularities, and let $y\in D^2$ be a regular value.
A \emph{Hurwitz system} for~$f$ based at~$y$ is a collection of embedded arcs $\eta_0,\eta_1,\dots,\eta_\ell\subset D^2$ such that
\begin{enumerate}
\item
$\eta_0$ connects~$y$ to~$\del D^2$ and is transverse to~$\del D^2$,
\item
$\eta_i$ connects~$y$ to a critical value~$y_i$,
\item
the arcs intersect pairwise transversely in~$y$ and are otherwise disjoint, and
\item 
the order of the arcs is counterclockwise around~$y$.
\end{enumerate}
\end{definition}

Given a Hurwitz system, we obtain a collection of simple closed curves in the regular fibre $F_y=f\inv(y)$ as follows.
%
For $i>0$ we have the classical construction of \emph{Lefschetz vanishing cycles}:
as we move from $y$ along~$\eta_i$ towards~$y_i$, a curve $\lambda_i\subset F_y$ shrinks and eventually collapses into Lefschetz singularity over~$y_i$, leading to a nodal singularity in~$F_{y_i}$. 
For later reference, we also recall that the monodromy along a counterclockwise loop around~$y_i$ contained in a neighbourhood of~$\eta_i$ is given by a right-handed Dehn twist about~$\lambda_i$.
%
Along $\eta_0$ we see a slightly different degeneration:
the boundary of a solid torus degenerates the core circle.
Indeed, using the local model for~$f$ near~$\mc D$ and the transversality of $\eta_0$ to~$\del D^2$ we can find a diffeomorphism~$f\inv(\eta_0)\cong D^2\times S^1$ and a parameterization of~$\eta_0$ that takes~$f$ into the function~$D^2\times S^1 \to \R \times Z$ given by $(x_1,x_2,\theta)\mapsto (x_1^2+x_2^2,z_0)$, where $z_0 = \eta_0(1)$. To summarize, $f\inv(\eta_0)$ is a solid torus whose boundary is~$F_y$. Further~$F_y$ contains a well-defined isotopy class of meridional circles, represented in the model by $\del D^2\times\{\theta\}$ for arbitrary~$\theta\in S^1$.
We will henceforth refer to this isotopy class as the \emph{boundary vanishing cycle} associated to~$\eta_0$ and denote it by~$\delta$.

To make things even more concrete, we can fix an identification of the reference fibre~$F_y$ with~$T^2$ and consider the vanishing cycles in the standard torus.
To make a notational distinction, we denote the images in $T^2$ by $(a;b_1,\dots,b_\ell)$.

\begin{definition}[Cycle systems]\label{D:cycle systems}
A collection of curves $(a;b_1,\dots,b_\ell)$ in~$T^2$ associated to~$f$ by a choices of a Hurwitz system and an identification of the reference fibre with~$T^2$ is called a \emph{cycle system} for~$f$.
\end{definition}

It is well known that the Lefschetz part of~$f$ can be recovered from the Lefschetz vanishing cycles.
In the next section we will explain how this statement extends to \blfs.
Just as in the Lefschetz case, the cycle system is not unique but the ambiguities are easy to understand and provide some flexibility to find particularly nice cycle systems representing a given \blf.
The following is a straightforward generalization of the analogous statement for Lefschetz fibrations, see also \autoref{fig:change of base}.

\begin{proposition}\label{T:Hurwitz equivalence} Let $f\colon (X^4,\mathcal{D}^2)\ra (D^2,\del D^2)$ be a \blf with~$\ell$ Lefschetz singularities. Any two cycle systems for~$f$ are related by a finite sequence of the following modifications:
	\begin{align*}
	\big( a ; b_1,\dots,b_\ell \big), &\sim
	\big( a ; b_2,B_2(b_1),b_3,\dots,b_\ell \big), \\ &\sim
	\big( a ; B_1\inv(b_2),b_1,b_3,\dots,b_\ell \big), \\ &\sim
	\big( B_1(a) ; b_2,\dots,b_\ell,b_1 \big), \\&\sim
	\big( B_\ell\inv(a) ; b_\ell,b_1,\dots,b_{\ell-1} \big), \\&\sim
	\big( h(a) ; h(b_1),\dots,h(b_\ell) \big).
	\end{align*}
Here $B_i=\tau_{b_i}$ is a right-handed Dehn twist about~$b_i$ and $h$~is any diffeomorphism of~$T^2$.
\end{proposition}

\begin{figure}[h!!]
\begin{center}
\includegraphics[height=8.5cm]{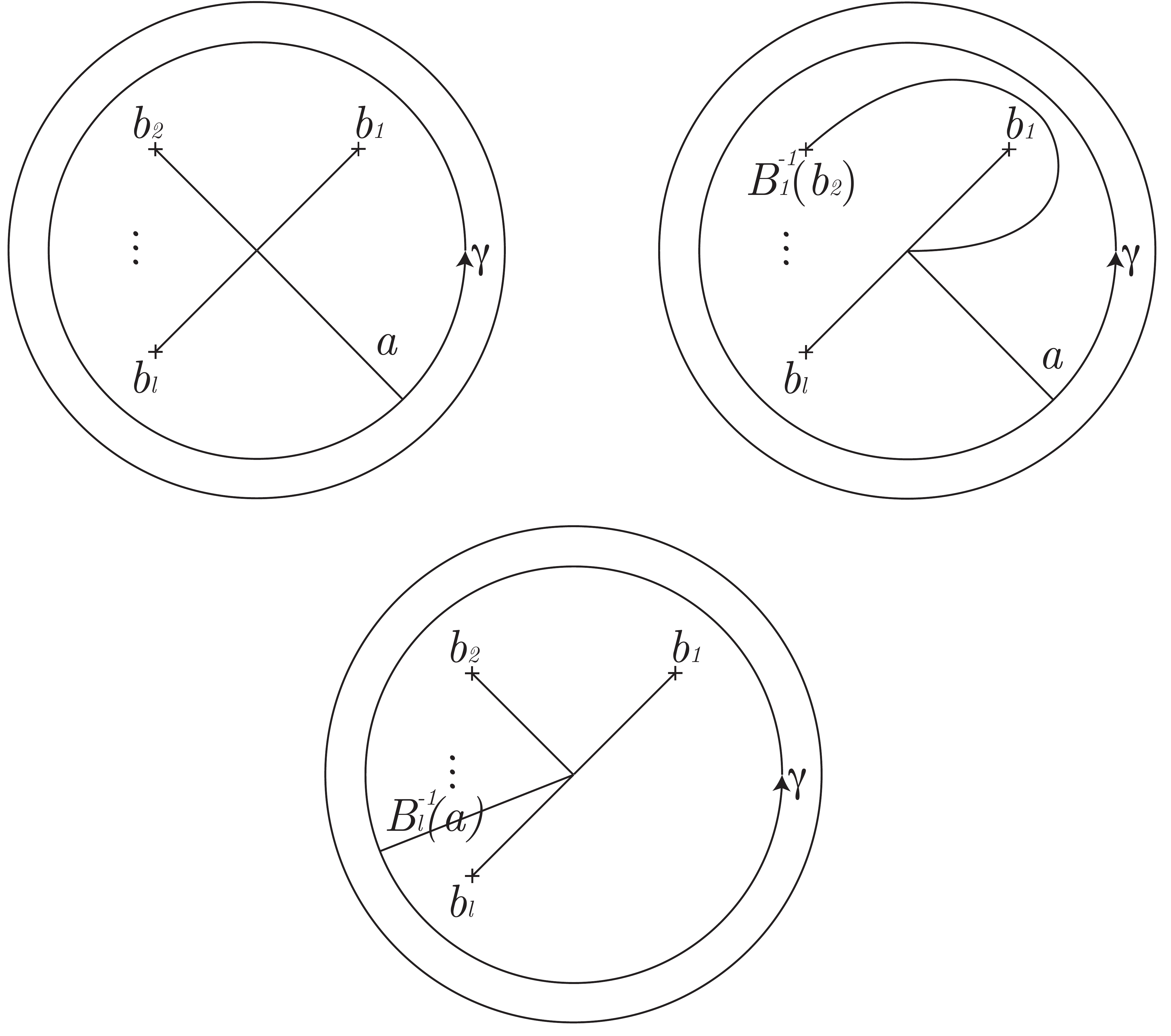} \caption{The origin of Hurwitz equivalence: here we illustrate how the equivalences $\big( a ; b_1,\dots,b_\ell \big) \sim ( a ; B_1\inv(b_2),b_1,b_3,\dots,b_\ell \big)\sim  \big( B_\ell\inv(a) ; b_\ell,b_1,\dots,b_{\ell-1} \big)$ arise. } \label{fig:change of base}
\end{center}
\end{figure}

\begin{definition}[Hurwitz equivalence]\label{D:Hurwitz equivalence}
If two cycle systems are related by the modifications listed in \autoref{T:Hurwitz equivalence}, we say that they are \emph{(Hurwitz) equivalent}.
\end{definition}

It turns out that the curves in a cycle system are not completely arbitrary.
Let $S^1_r\subset D^2$ be the circle of radius~$r<1$ such that all the Lefschetz singularities of~$f$ map to the interior of~$D^2_r$.
Fix a reference point let~$y\in S^1_r$ and let 
	\begin{equation*}
	\mu(f) \in \MCG(F_y)=\pi_0\mathrm{Diff}^+(F_y)
	\end{equation*}
be the counterclockwise monodromy of~$f$ around~$S^1_r$ as measured in the mapping class group of~$F_y$.
Then for any cycle system for~$f$ derived from a Hurwitz system based at~$y$ the counterclockwise monodromy of~$f$ around~$S^1_r$ measured in~$F_y$ is given by the product of Dehn twists about the Lefschetz vanishing cycles~$\lambda_i\subset F_y$,
	\begin{equation}\label{eq:monodromy}
	\mu(f) = \tau_{\lambda_\ell}\circ\dots\circ \tau_{\lambda_1} \in \MCG(F_y)=\pi_0\mathrm{Diff}^+(F_y).
	\end{equation}
On the other hand, we can also describe the monodromy using the boundary part of the fibration.
Recall that~$f\inv(S^1_r)$ is the boundary of a tubular neighbourhood~$N\mc D$ of~$\mathcal{\mc D}$ and that the fibration structure over~$S^1_r$ essentially factors through the projection~$N\mc D\ra \mc D$.
This exhibits $f\inv(S^1_r)$ as a circle bundle over~$\mc D$, which is itself a circle bundle over~$S^1$.
It follows that the monodromy of~$f$ around~$S^1_r$ must fix the circle fibres of~$f\inv(S^1_r)\ra\mc D$, and the circle fibre contained in~$F_y$ is precisely the boundary vanishing cycle~$\delta$ of the Hurwitz system. 
To conclude, $\mu(f)$ fixes~$\delta$ as a set, but not necessarily pointwise.
Indeed, it can (and does) happen that $\mu(f)$ reverses the orientation of~$\delta$.

\begin{remark}\label{R:the torus rules}
At this point, it is worthwhile to point out some perks of working on a torus.
First, there is the fact that any diffeomorphism of~$T^2$ is determined up to isotopy by its action on~$H_1(T^2)$.
Given any pair of oriented simple closed curves~$a,b\subset$ with (algebraic) intersection number~$\scp{a,b}=1$ --- called \emph{dual pairs} from now on --- we get an identification $\MCG(T^2)\cong SL(2,\Z)$.
Moreover, the right-handed Dehn twists $A,B\in\MCG(T^2)$ about~$a$ and~$b$ are the generators in a finite presentation with relations $ABA=BAB$ and $(AB)^6=1$.
In particular, we have that $(AB)^3$ maps to~$-1\in SL(2,\Z)$, which we will also denote by writing $-1=(AB)^3\in\MCG(T^2)$.
Second, in a similar fashion, simple closed curves up to ambient isotopies are uniquely determined by their (integral) homology classes.
Note that this involves a choice of orientation, since simple closed curves are a priori unoriented objects.
However, it is true that essential simple closed curves in~$T^2$ correspond bijectively with primitive elements of~$H_1(T^2)$ up to sign.
In what follows we adopt the common bad habit of identifying simple closed curves with elements of~$H_1(T^2)$ without explicitly mentioning orientation.
In particular, we will freely use the homological expression for a Dehn twist, i.e.\ write
	\begin{equation}\label{eq:Picard-Lefschetz}
	\tau_c(d) = d + \scp{c,d} c \in H_1(T^2).
	\end{equation}
We record two facts that are important for our purposes:
\begin{enumerate}
\item
If $h\in\MCG(T^2)$ satisfies $h(a)=a$ for some essential curve~$a$, then $h=\pm\tau_a^k$ for some~$k$ with a negative sign if and only if $h$ is orientation-reversing on~$a$;
\item 
If oriented curves~$a,b,c\subset T^2$ satisfy $\scp{a,b}=\scp{a,c}=1$, then $c=\tau_a^k(b) = b+ka$.
\end{enumerate}
\end{remark}

Returning to the discussion of the monodromy~$\mu(f)$, we can conclude that the vanishing cycles have to satisfy the condition
	\begin{equation*}
	\mu(f) = \tau_{\lambda_\ell}\circ\dots\circ \lambda_{\lambda_1} = \pm\tau_\delta^k \in\MCG(F_y).
	\end{equation*}
It is easy to see from the above discussion that a negative sign appears if and only if $\mc D$ fails to be co-orientable.
Moreover, the integer~$k$ is precisely the Euler number of the normal bundle of~$\mc D$ in~$X$. Here we remark that a vector bundle $E\ra M$ with $M$ compact has a well-defined integer Euler number if the total space of $E$ is orientable, even if~$M$ is not orientable itself.

For practical purposes, it is more convenient to work with cycle systems in the model~$T^2$.
Here is the upshot of the above discussion:

\begin{proposition}\label{T:compatibility}
Let $f\colon (X^4,\mc D^2)\ra (D^2,\del D^2)$ be a \blf.
If $(a;b_1,\dots,b_\ell)$ is any cycle system for~$f$, then
	\begin{equation}\label{eq:compatibility condition}
	B_\ell\circ\dots\circ B_1 = \pm A^k \in \MCG(T^2)
	\end{equation}
for some~$k\in \Z$, where the sign is positive if and only if~$\mc D$ is co-orientable.
The integer~$k$ agrees with the Euler number of the normal bundle of~$\mc D$ in~$X$.
\end{proposition}

This motivates an abstract definition without reference to \blfs.

\begin{definition}[Abstract cycle systems]\label{D:abstract cycle systems}
An ordered collection of curves $(a;b_1,\dots,b_\ell)$ in~$T^2$ is called an \emph{abstract cycle system} if it satisfies the condition in~\eqref{eq:compatibility condition}.
The notion of \emph{Hurwitz equivalence} is defined exactly as in \autoref{D:Hurwitz equivalence}.
\end{definition}


\subsection{Handle decompositions and Kirby diagrams}
\label{ch:handle decompositions}

Next we discuss how to recover \blfs from their cycle systems.
Along the way, we exhibit useful handle decompositions of total spaces of \blfs. 

\begin{proposition}\label{T:building bLFs}
Any abstract cycle system $(a;b_1,\dots,b_\ell)$ is the cycle system of some \blf over the disc.
\end{proposition}

\begin{proof}
We will build a four-manifold obtained by attaching handle to~$T^2\times D^2$.
We choose points $\theta_0,\dots,\theta_\ell\in \del D^2$ which appear in counterclockwise order and consider a copy of~$a$ in~$T^2\times\{\theta_0\}$ and of~$b_i$ in~$T^2\times\{\theta_i\}$ for~$i>0$.
Note that for all these curves there is a natural choice of framing determined by parallel push-offs inside the fibres of~$T^2\times S^1\ra S^1$.
We first attach $2$-handle along the copies of~$b_i$ for $i>0$ with respect to the fibre framing~$-1$ and call the resulting manifold~$Z$.
It is well known that the projection~$T^2\times D^2$ extends to a Lefschetz fibration on~$Z$ over a slightly larger disc, which we immediately rescale to~$D^2$, such that the Lefschetz vanishing cycles along the straight line from~$\theta_i$ to zero is~$b_i$.
By construction, the boundary fibres over~$S^1$ and the counterclockwise monodromy measured in~$T^2\times\{\theta_0\}$ is $B_\ell\circ\dots\circ B_1=\pm A^k$.
In particular, $\del Z$ is diffeomorphic as an oriented manifold to the circle bundle with Euler number~$k$ over the torus or the Klein bottle.
Let $N_{-k}^\pm$ be the corresponding disc bundle with Euler number~$-k$.
Then $\del N_{-k}^\pm$ is diffeomorphic to~$\del Z$ with the orientation reversed so that we can form a closed manifold~$X$ by gluing~$Z$ and~$N_{-k}^\pm$ together, and the orientation of~$Z$ extends.
Moreover, it was shown in~\cite{2017arXiv170303798C} that $N_{-k}^\pm$ admits a boundary fibration over the annulus which can be used to extend the Lefschetz fibration on~$Z$ to a \blf on~$X$, again over a larger disc which we recale to~$D^2$, in such a way that the boundary vanishing cycle along the straight line from~$\theta_0$ to zero is~$a$.

Thus we have found a \blf together with a Hurwitz system which produces the desired cycle system.
\end{proof}
\begin{remark}[Construction of the Kirby diagram]\label{R:Kirby}
Observe that the gluing of~$N_{-k}^\pm$ also has an interpretation in terms of handles.
It is well known that $N_{-k}^\pm$ has a handle decomposition with one $0$-handle, two $1$-handles, and a single $2$-handle.
Turning this decomposition upside down gives a relative handle decomposition on~$-\del N_{-k}^\pm\cong \del Z$ with a single $2$-handle, two $3$-handles, and a $4$-handle.
Moreover, the $2$-handle can be chosen such that its core disc is a fibre. 
In particular, since the gluing of $N_{-k}^\pm$ to~$Z$ preserves the circle fibration, can arrange that the $2$-handle of~$N_{-k}^\pm$ is attached along the copy of~$a$ in the fibre of~$\del Z$ over~$\theta_0$.
However, in contrast to the Lefschetz handles, this time the framing is actually the fibre framing.

To summarize, the closed four-manifold~$X$ is obtained from $T^2\times D^2$ by attaching, in order, a $2$-handle along the boundary vanishing cycle with the  fibre framing, and then $2$-handles along the Lefschetz vanishing cycles $b_i\subset T^2\times\{\theta_i\}$ with fibre framing~$-1$. The two $3$-handles as well as the $4$-handle attach uniquely by Laudenbach--Po\'{e}naru.

As an illustration of this procedure, \autoref{fig:aa} shows the Kirby diagrams corresponding to the abstract cycle systems $(a;a)$ and $(a;b+2a,b)$, where $\{a,b\}$ is a dual pair of curves.
\begin{figure}[h!!]
\begin{center}
\includegraphics[height=4.7cm]{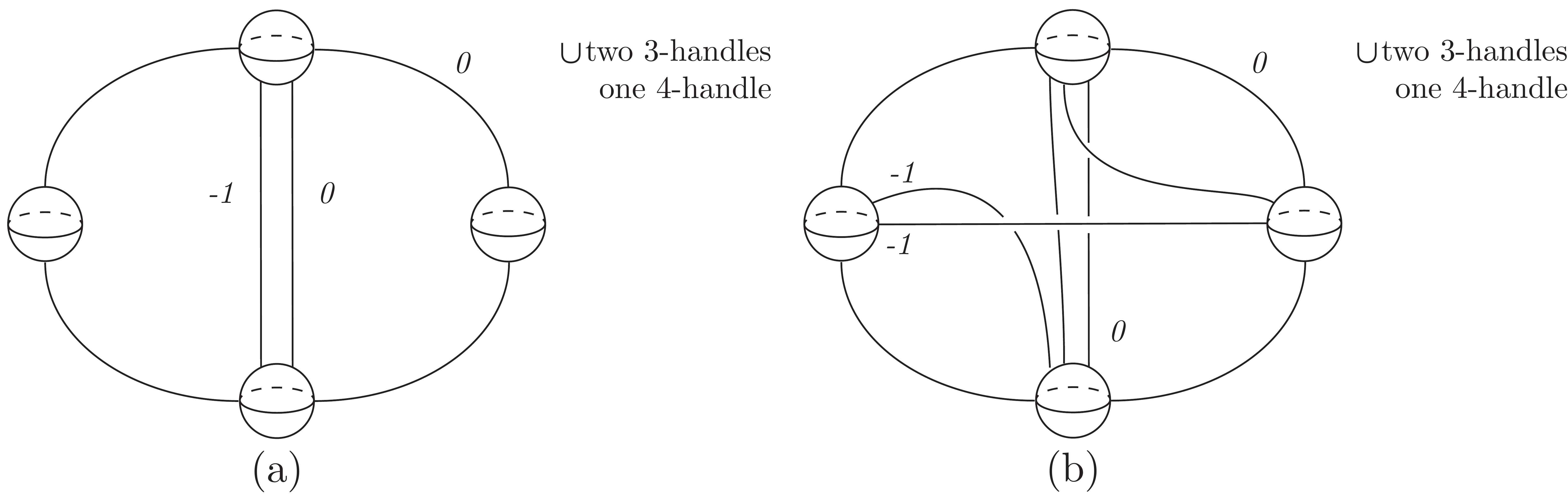} \caption{Kirby diagrams corresponding to the abstract cycle systems $(a;a)$ (Figure (a)) and $(a;b+2a,b)$ (Figure (b)). The numbers indicate the blackboard framing of the corresponding $2$-handles.} \label{fig:aa}
\end{center}
\end{figure}
\end{remark}

Next we show that the topology of the total space of a \blf can be recovered from the cycle system.
\begin{proposition}\label{T:recovering bLFs}
If two \blfs over the disc have equivalent cycle systems, then their total spaces are diffeomorphic.
\end{proposition}
\begin{proof}
Elaborating on the proof of \autoref{T:building bLFs}, one can show that, if a Hurwitz system and identification of the reference fibre with~$T^2$ of a \blf $f\colon (X^4,\mathcal{D}^2)\ra (D^2,\del D^2)$ produces the cycle system $(a;b_1,\dots,b_\ell)$, then $X$ is diffeomorphic to the manifold constructed by attaching handles to~$T^2\times D^2$ as explained above.
Similarly, one can then argue that the manifolds constructed from equivalent cycle systems are diffeomorphic.
The details are somewhat tedious but straightforward and we leave them to the inclined reader.
\end{proof}

As a consequence, in order to classify closed four-manifolds admitting \blfs over~$D^2$, it is enough to identify all four-manifolds obtained from abstract cycle systems as in the proof of \autoref{T:building bLFs}.
Moreover, as we argued in \autoref{R:Kirby}, this problem is naturally accessible to the methods of Kirby calculus via the handle decompositions.
For the relevant background about Kirby calculus we refer to~\cite{GompfStipsicz} (Chapter~8, in particular).

\section{Boundary Lefschetz fibrations over \texorpdfstring{$D^2$}{D2}}\label{sec:blfs over D2}

As a warm-up to our main theorem, it is worth considering the following more basic question: Which oriented four-manifolds are boundary fibrations over $D^2$? The answer is very simple:

\begin{lemma}\label{T:no Lefschetz}
Let $X$ be a compact, orientable manifold and let $f\colon (X^4,\mathcal{D}^2) \to (D^2, \del D^2)$ be a boundary fibration. Then $X$ is diffeomorphic to $S^1 \times S^3$ and $\mc D$ is co-orientable.
\end{lemma}
\begin{proof}
Note that a boundary fibration is a \blf without Lefschetz singularities.
As such, its cycle systems consist of a single curve~$a\subset T^2$ corresponding to the boundary vanishing cycle. 
Thus $a$~is essential and we can therefore assume that~$a = \{1\}\times S^1$.
According to the discussion in \autoref{ch:handle decompositions}, $X$ is obtained from gluing $T^2\times D^2$ together with a suitable disc bundle over a torus or Klein bottle, such that the boundary of a disc fibre is identified with~$a$.
Obviously, the only possibility is $N_0^+ = D^2\times T^2$, the trivial disc bundle over the torus,
and the gluing can be arranged such that $\del D^2\times \{(1,1)\}\subset N_0^+$ is identified with $a\times \{1\}\subset T^2\times D^2$.
Since this is achieved by the diffeomorphism of $T^3$ which flips the first two factors, we see that
	\begin{align*}
	X 
	&\cong S^1\times S^1\times D^2 \cup_\varphi D^2\times S^1\times S^1 \\
	&\cong S^1\times S^1\times D^2 \cup_{\mathrm{id}} S^1\times D^2\times S^1\\
	&\cong S^1\times \big(S^1\times D^2 \cup_{\mathrm{id}} D^2\times S^1\big)
	\cong S^1\times S^3,
	\end{align*}
where the last diffeomorphism comes from the standard decomposition of~$S^3$ considered as sitting in~$\C^2$ and split into two solid tori by $S^1\times S^1\subset\C^2$.
\end{proof}

Now we move on to study honest boundary Lefschetz fibrations over the disc and eventually prove our classification theorem, \autoref{theo:blf over D2}. The proof of the theorem itself is done by induction on the number of singular fibres. So, in order to achieve our aim, we need to study the base cases, i.e., boundary Lefschetz fibrations with only a few singular fibres, and explain how to systematically reduce the number of singular fibres to bring us back to the base cases. It turns out that there is a step that appears frequently, namely, the blow-down of certain $(-1)$-spheres which is interesting on its own as it gives the notion of a relatively minimal boundary Lefschetz fibration. In the rest of this section, we will first study blow-downs and relatively minimal fibrations. We then move on to study the cases with one and two singular fibres and finally prove \autoref{theo:blf over D2}.

\subsection{The blow-down process and relative minimality}

Given a usual Lefschetz fibration $f\colon X^4 \to \Sigma^2$, we can perform the blow-up in a regular point~$x\in X$ with respect to a local complex structure compatible with the orientation of~$X$.
The result is a manifold $\wt X$ together with a blow-down map~$\sigma\colon \wt X\ra X$ and it turns out that the composition $\wt f =f\circ\sigma\colon \wt X\ra \Sigma$ is a Lefschetz fibration with one more critical point than~$f$ in the fibre over~$y=f(x)$.
Moreover, the exceptional divisor sits inside the (singular) fibre $\wt f\inv(y)$ as a sphere with self-intersection~$-1$.
Conversely, given any $(-1)$-sphere in a singular fibre of a Lefschetz fibration this process can be reversed: the $(-1)$-sphere can be blown down producing a Lefschetz fibration with one critical point less.
For that reason it is enough to study \emph{relatively minimal} Lefschetz fibrations: fibrations whose fibres do not contain any $(-1)$-spheres.
Equivalently, a Lefschetz fibration is relatively minimal if no vanishing cycle bounds a disc in the reference fibre; 
and on the level of cycle systems the blow-up and blow-down procedures simply amount to adding or removing null-homotopic vanishing cycles.

For a boundary Lefschetz fibration $f\colon(X^4,\mathcal{D}^2) \to (\Sigma^2,Z^1)$ there is another way a $(-1)$-sphere can occur in relation to the fibration. 
These spheres arise if there is a simple path connecting a Lefschetz singular value of~$f$ to a component of~$Z$ with the property that the Lefschetz vanishing cycle in one end of the path agrees with the boundary vanishing cycle. In this case, we can form the corresponding Lefschetz thimble from the Lefschetz singularity which then closes up at the other end of the path to give rise to a $(-1)$-sphere, $E$, which intersects the divisor $\mathcal{D}$ at one point, as observed in \cite{MR2574746}.
Observe that, in the case where $(\Sigma,Z)=(D^2,\del D^2)$ this is equivalent to a cycle system~$(a;b_1,\dots,b_\ell)$ such that some~$b_i$ agrees with~$a$.
From this description, it is clear that we can blow $E$ down to obtain a new manifold, $X'$. What is not immediately clear is that $X'$ admits the structure of a boundary Lefschetz fibration. 

\begin{proposition}\label{T:blow-down}
Let $f\colon (X^4,\mc D^2)\ra (D^2,\del D^2)$ be a \blf. 
If $f$ has a cycle system~$(a;b_1,\dots,b_\ell)$ such that $b_i=a$ for some~$i$, then there exists another \blf $f'\colon (X',\mc D')\ra (D^2,\del D^2)$ with cycle system~$\big(a;A(b_1),\dots,A(b_{i-1}),b_{i+1},b_\ell\big)$, where $A$ denotes a Dehn Twist about $a$.
Moreover, we have $X\cong X'\#\overline{\C P^2}$ and~$\mc D'$ has the same co-orientability  as $\mc D$.
\end{proposition}
\begin{proof}
This is our first exercise in Kirby calculus.
Using Hurwitz moves we have the equivalence of cycle systems:
	\begin{equation*}
	\big (a;b_1,\dots,b_{i-1},a,b_{i+1},\dots,b_\ell \big)
	\cong \big(a;a, A(b_1),\dots,A(b_{i-1}),b_{i+1},b_\ell\big),
	\end{equation*}
so we may assume without loss of generality that $b_1 =a$. Further, we can take $a$ to be the first cycle of a dual pair $\{a,b\}$, that is, we may assume that~$a=S^1\times\{1\}\subset T^2$.
We now compare the Kirby diagrams obtained from the cycle systems $(a;a,b_2,\dots,b_\ell)$ and~$(a;b_2,\dots,b_\ell)$.

As we mentioned in \autoref{R:Kirby}, to draw a Kirby diagram for a boundary Lefschetz fibration corresponding to a cycle system, we start with the Kirby diagram of $D^2 \times T^2$ and add cells corresponding to the boundary vanishing cycle followed by the Lefschetz vanishing cycles ordered counterclockwise. Therefore, the Kirby diagram for $(a;a,b_2,\dots,b_\ell)$ is the Kirby diagram for $(a;a)$ with a number of $2$-handles on top of it representing the  cycles $b_2,\dots,b_\ell$.  The Kirby move we use next does not interact with these last $(l-1)$ $2$-handles, therefore we will not represent them in the diagram. With this in mind, the relevant part of the Kirby diagram of $(a;a,b_2,\dots,b_\ell)$ is the Kirby diagram of $(a;a)$ as drawn in \autoref{fig:aa}.(a). Sliding the $-1$-framed $2$-handle corresponding to the first Lefschetz singularity over the $0$-framed $2$-handle corresponding the boundary vanishing cycle produces a $-1$-framed unknot which is unlinked from the rest (see \autoref{fig:blowdown}).
The remaining Kirby diagram is precisely that corresponding to the cycle system~$(a;b_2,\dots,b_\ell)$.
Since an isolated $-1$-framed unknot represents a connected sum with~$\overline{\C P^2}$, the result follows. \qedhere
\begin{figure}[h!!]
\begin{center}
\includegraphics[height=4.7cm]{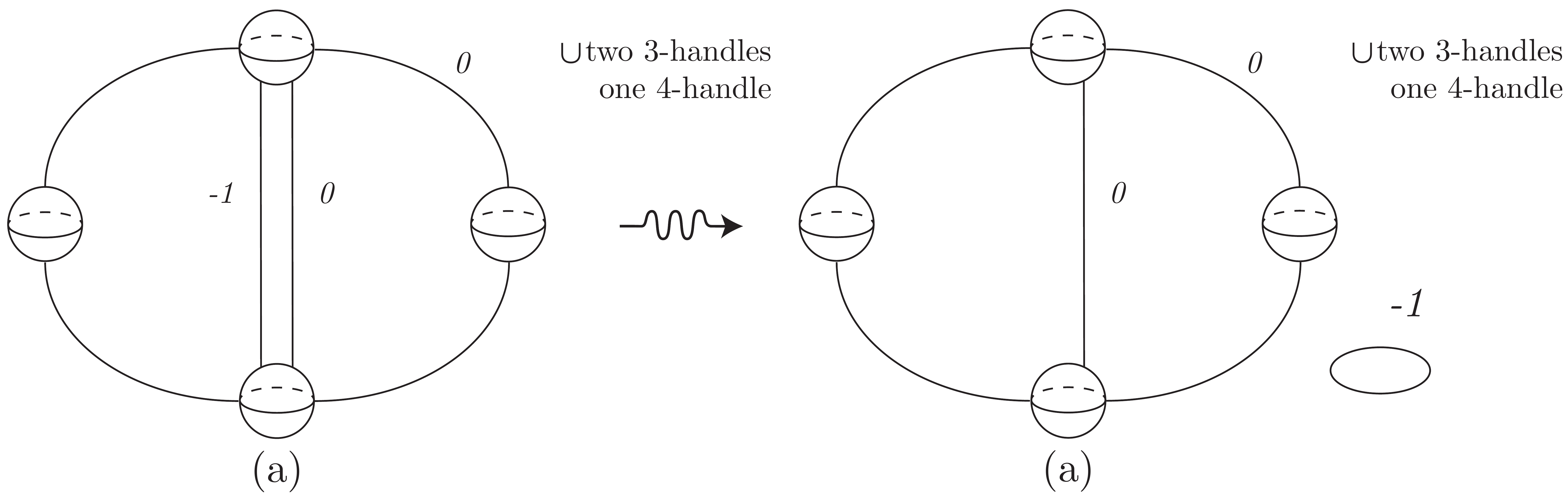} \caption{Figure showing the relevant part of the Kirby diagram of the cycle system $(a;a,b_2,\dots, b_\ell)$ and the result of sliding the $-1$-framed $2$-handle over the $0$-framed one.} \label{fig:blowdown}
\end{center}
\end{figure}
\end{proof}

The previous proof is prototypical for much of what follows from now on.
In light of \autoref{T:blow-down} we make the following definition.

\begin{definition}[Relative minimality]\label{D:relatively minimal}
A \blf $f\colon (X^4,\mc D^2)\ra (D^2,\del D^2)$ is called \emph{relatively minimal} if there is no cycle system $(a;b_1,\dots,b_\ell)$ for~$f$ in which some Lefschetz vanishing cycle~$b_i$ is either null-homotopic or parallel to~$a$.
\end{definition}

\subsection{Boundary Lefschetz fibrations over \texorpdfstring{$D^2$}{D2} with few singular fibres}

The next step is to determine which manifolds admit boundary Lefschetz fibrations with only one or two singular Lefschetz fibres.

\begin{lemma}\label{T:one Lefschetz}
Let $f\colon(X^4,\mathcal{D}^2) \to (D^2, \del D^2)$ be a boundary Lefschetz fibration with a single singular Lefschetz fibre. Then $f$ is not relatively minimal, we have $X \cong (S^1 \times S^3)\# \overline{\C P^2}$, and $\mc D$ is co-orientable.
\end{lemma}
\begin{proof}
This is \cite[Example 8.4]{2017arXiv170303798C}, but in light of our discussion about blow-ups in terms of cycle systems we can determine it directly. 
Indeed, any cycle system of~$f$ has the form~$(a;b_1)$ such that $B_1=\pm A^k$.
Clearly this is only possible when~$b_1$ is either null-homotopic or parallel to~$a$.
In either case, $f$ is not relatively minimal and can be blown down to a boundary fibration, which, by \autoref{T:no Lefschetz}, is diffeomorphic to $S^1 \times S^3$.
\end{proof}


\begin{lemma}\label{T:two Lefschetz}
Let $f\colon(X^4,\mathcal{D}^2) \to (D^2, \del D^2)$ be a relatively minimal boundary Lefschetz fibration with two singular Lefschetz fibres. 
Then $X \cong S^4$ and $\mc D$ is not co-orientable.
\end{lemma}
\begin{proof}
All cycle systems of~$f$ have the form~$(a;b_1,b_2)$ with~$b_1$ and~$b_2$ essential and not parallel to~$a$.
At this level of difficulty one can still perform direct computations. This was done by Hayano in \cite{MR2801419}. The outcome is that we must have $b_1=A^2(b_2)=b_2+2a$ and for suitable orientations we have $\langle a,b_1\rangle=\langle s,b_2\rangle=1$. 
Using the relation $AB_2A=B_2AB_2$ and $(AB_2)^3=-1$ in~$\MCG(T^2)$ we find that
	\begin{equation*}
	\mu(f) 
	= B_2B_1
	= B_2 A^{2}B_2A^{-2}
	=B_2 A (A B_2A )A A^{-4}
	=B_2A (B_2A B_2)A A^{-4}
	= -A^{-4}.
	\end{equation*}
The corresponding Kirby diagram for~$X$ is given in \autoref{fig:aa}.(b).

This particular type of Kirby diagram will appear repeatedly in this paper, so we deal with it in a separate claim.

Just as we mentioned in the proof of \autoref{T:blow-down}, when drawing the Kirby diagram for a boundary Lefschetz fibration we must draw, from bottom to top, a $0$-framed $2$-handle corresponding to the boundary vanishing cycle and then $-1$-framed $2$-handles for each Lefschetz singularity ordered counterclockwise. We will often want to make simplifications to the diagram which involve only the bottom two or three $2$-handles.

\begin{lemma}\label{lem:basickirby}
Let $a,b\subset T^2$ be a dual pair of curves. Then the Kirby diagram associated to a cycle system of the form $(a;A^k(b),b,\dots)=(a;b+ka,b,\dots)$ is equivalent to that in \autoref{fig:basickirby}. 
\begin{figure}[h!!]
\begin{center}
\includegraphics[height=4.7cm]{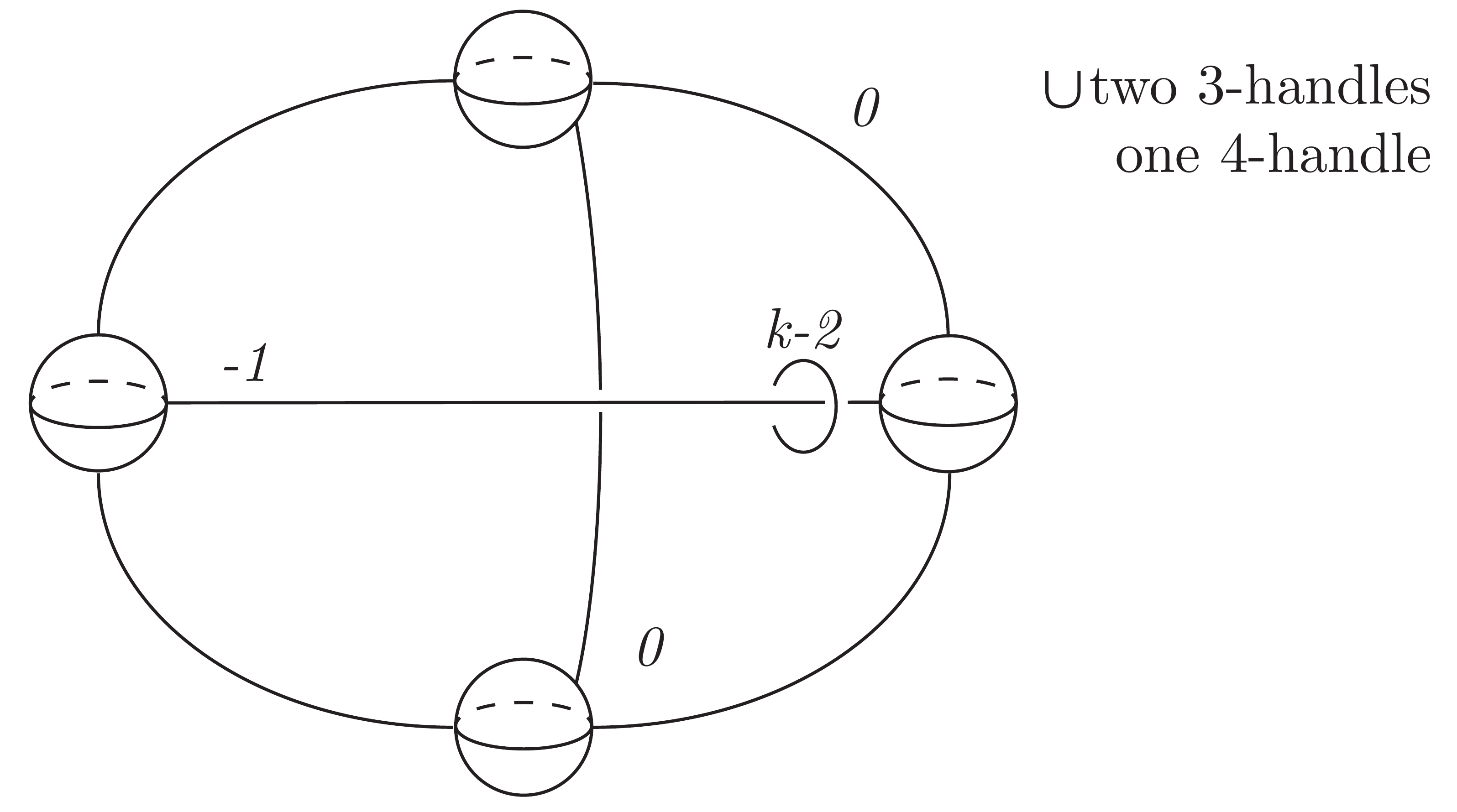} 
\caption{A Kirby diagram for cycle systems $(a;b+ka,b,\dots)$ after handle slides.
Only the first three $2$-handles are shown, the other handles appear above the diagram in their standard form. In particular, they are unlinked from the $(k-2)$-framed unknot.} \label{fig:basickirby}
\end{center}
\end{figure}
\end{lemma}
\begin{proof}
The proof is a simple exercise: slide the $2$-handle corresponding to $b+ka$ $k$ times over the $0$-framed $2$-handle representing $a$ and once over the $2$-handle corresponding to $b$. None of these manoeuvres interacts with the other handles. 
\end{proof}

\noindent
{\it Proof of \autoref{T:two Lefschetz} continued.} Using \autoref{lem:basickirby} we see that the boundary Lefschetz fibration is equivalent to the one depicted in \autoref{fig:basickirby} with $k=2$. If we slide the outer $2$-handle over the `$a$-handle' twice we get the diagram depicted in \autoref{fig:S4b}. There, a few things happen: the outer $0$-framed $2$-handle can be pushed out of the $1$-handle and cancels a $3$-handle. The `$a$-handle' cancels one of the $1$-handles, and the `$b$-handle' cancels the other so we are left with a $0$-framed unknot which cancels the remaining $3$-handle. After all this cancellation we are left only with the $0$-handle and the $4$-handle, hence $X$ is $S^4$.
\begin{figure}[h!!]
\begin{center}
\includegraphics[height=4.7cm]{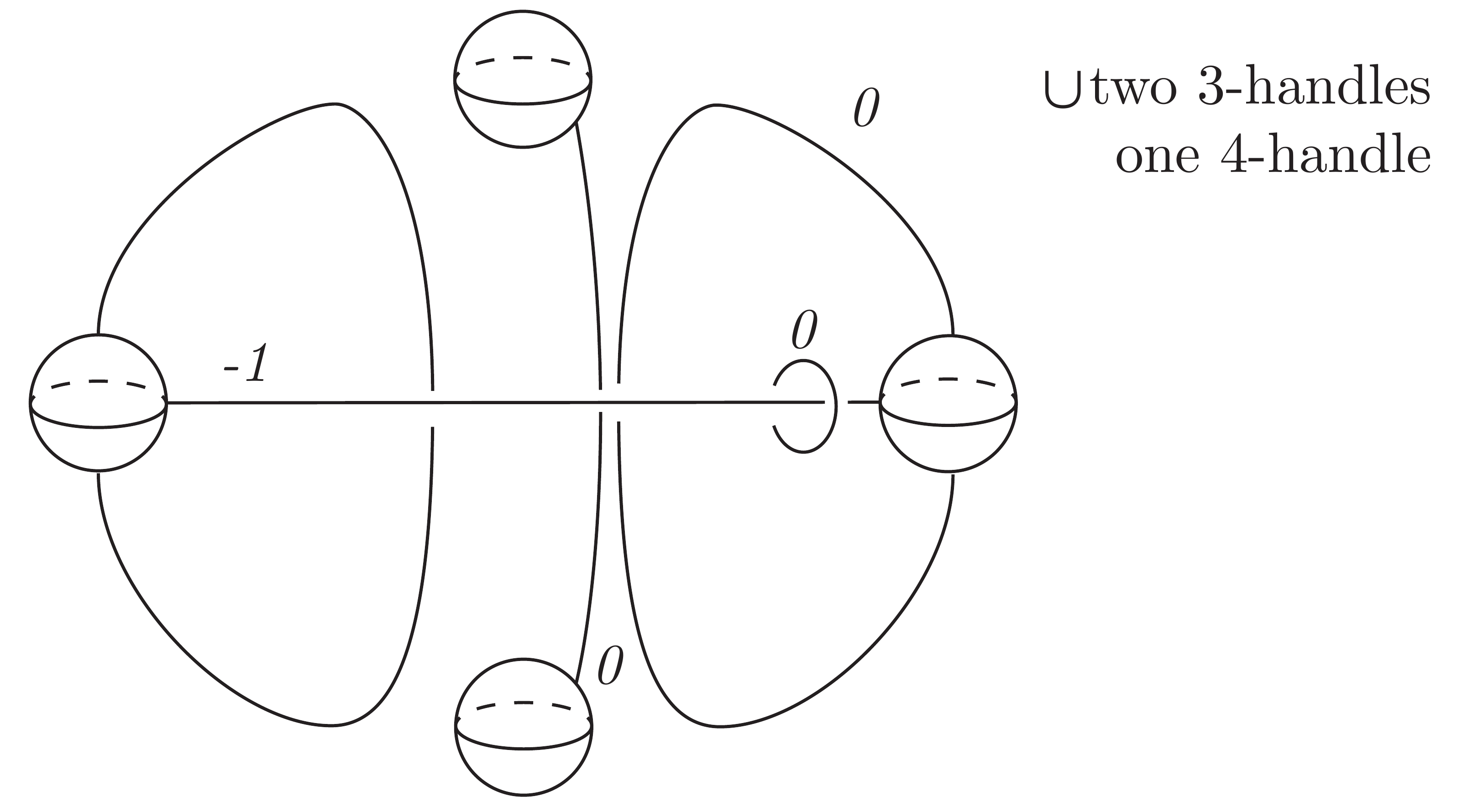} \caption{Kirby diagram for $X$ after two handle slides. Now everything cancels.} \label{fig:S4b}
\end{center}
\end{figure}
\end{proof}

\subsection{The inductive step}

The key for the induction are structural results about cycle systems of \blfs that were obtained by Hayano~\cite{MR2801419,MR3251824}, albeit in the slightly different but closely related context of genus-one simplified broken Lefschetz fibrations.
In what follows, $a,b\subset T^2$ is a fixed dual pair of curves, and $A,B\in\MCG(T^2)$ are the corresponding Dehn twists.

\begin{theorem}[Hayano Factorisation Theorem]\label{T:Hayano factorization}
Any abstract cycle system $(a;b_1,\dots,b_\ell)$ in the sense of \autoref{D:abstract cycle systems} is Hurwitz equivalent to one of the form
	\begin{equation}\label{eq:Hayano factorization}
	\big(a; \, a,\dots,a, \, b+k_1a,\dots,b+k_ra\big).
	\end{equation}
Moreover, for some~$1\leq i<r$ we must have $k_i-k_{i+1}\in\{1,2,3\}$.
\end{theorem}
\begin{proof}[Proof (by reference)]
This is a combination of \cite[Theorem~3.11]{MR2801419} and \cite[Lemma~4]{MR3251824}.
\end{proof}
\begin{figure}
\begin{center}
\includegraphics[height=4.7cm]{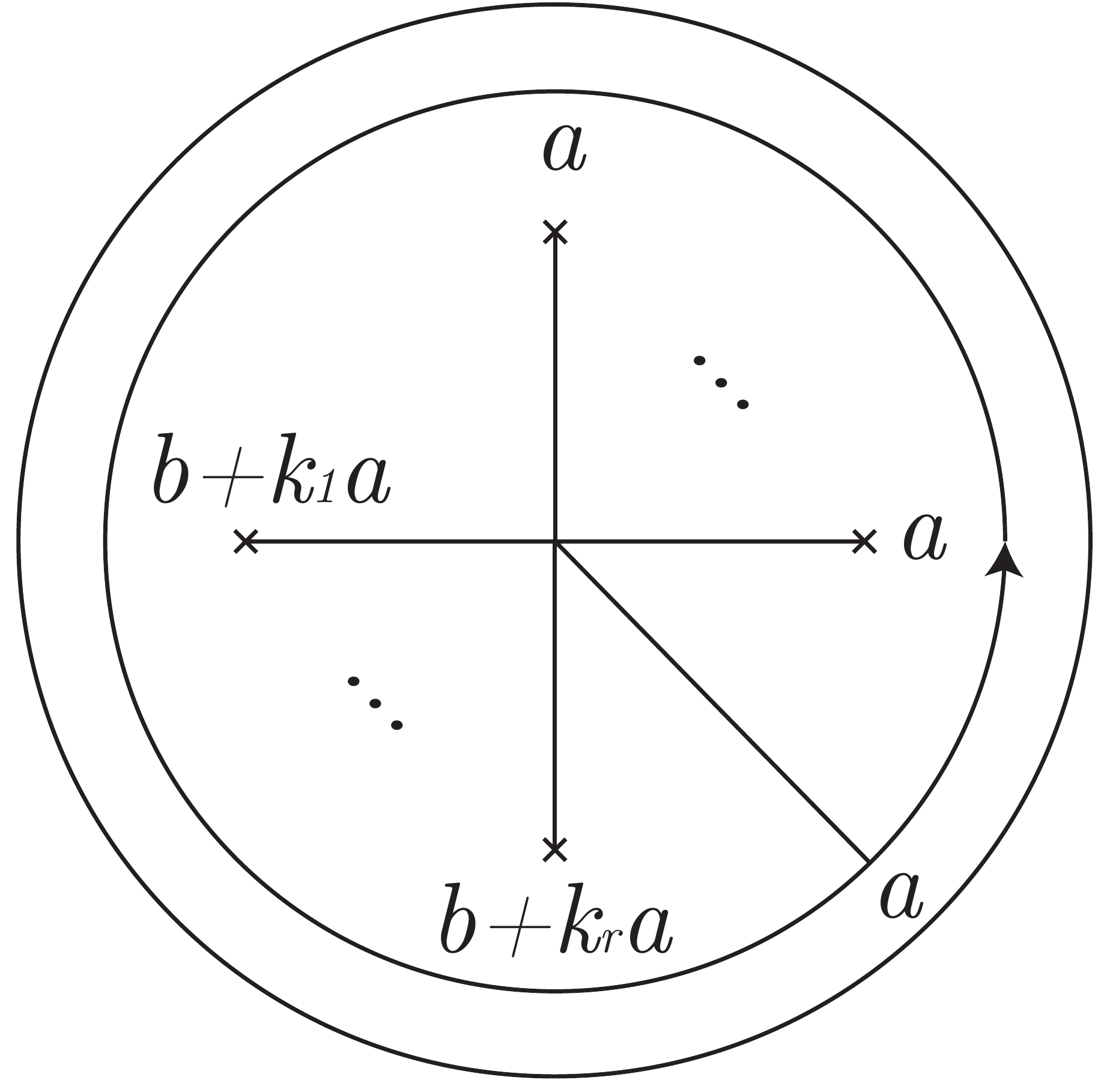} \caption{Factorisation for a boundary Lefschetz fibration from Hayano's theorem.} \label{fig:hayano1}
\end{center}
\end{figure}

As a consequence, any \blf over~$D^2$ admits a Hurwitz system as indicated in \autoref{fig:hayano1}.
Moreover, for relatively minimal fibrations we can say even more.

\begin{corollary}\label{T:factorization minimal case}
Let $f\colon (X^4,\mc D^2)\ra (D^2,\del D^2)$ be a relatively minimal \blf.
Then $f$ has a cycle system of the form
	\begin{equation}\label{eq:minimal normal form}
	\big(a;b+ka,b,b+na,\dots\big),
	\end{equation}
with $k\in\{1,2,3\}$ and $n\in\Z$.
\end{corollary}
\begin{proof}
From \autoref{T:Hayano factorization} and \autoref{T:blow-down} we can deduce that $f$ has a cycle system of the form 
	$\big(a;b+k_1a,\dots,b+k_ra\big)$ 
with $k_i-k_{i+1}\in\{1,2,3\}$ for some~$1\leq i<r$. 
By Hurwitz moves we can bring the cycle system to the form 
	$\big(a;b+k_i,b+k_{i+1}a,b+k_{i+2},\dots\big)$. 
Furthermore, by applying $A^{-k_{i+1}}$ we get
	$\big(a;b+(k_i-k_{i+1})a,b ,b+(k_{i+2}-k_{i+1})a,\dots\big)$. 
\end{proof}

The next step is to match the pattern in the cycle systems in~\eqref{eq:minimal normal form} with topological operations in the same spirit as \autoref{T:blow-down}.
This step is similar in form to what Hayano does while studying simplified broken Lefschetz fibrations (c.f. \cite[Theorem 4.6]{MR2801419}). We first treat the cases $k=1,3$.

\begin{proposition}\label{T:reduction 1 and 3}
Let $f\colon (X^4,\mc D^2)\ra (D^2,\del D^2)$ be a \blf over the disc with cycle system of the form $\big(a;b+ka,b, b_3,\dots,b_\ell\big)$ with $k\in\{1,3\}$.
\begin{enumerate}
\item
If $k=1$, then $f$ is not relatively minimal, that is, $X=X'\#\CPbar$ where $X'$ carries a \blf whose divisor has the same co-orientability as~$\mc D$;
\item 
If $k=3$, then there is a \blf $f'\colon (X',\mc D')\ra (D^2,\del D^2)$ with one fewer Lefschetz singularity.
We have~$X=X'\#\CP$ and the co-orientability of~$\mc D'$ is opposite to that of~$\mc D$.
A cycle system for~$f'$ is given by~$(a;b-a,b_3,\dots,b_\ell)$.
\end{enumerate}
\end{proposition}
\begin{proof}
For $k=1$ we have
$$(a;b+a,b,\dots)\sim (a;\tau_{a+b}^{-1}b, a+b,\dots) = (a;a,a+b,\dots)$$
by a single Hurwitz move, and we can then apply \autoref{T:blow-down}.

For $k=3$ we compare Kirby diagrams as in the proof of \autoref{T:blow-down}.
We can draw a Kirby diagram for this fibration in which we represent only the handles corresponding to $b+3a$ and $b$ and the boundary vanishing cycle and keep in mind that the handles corresponding to the other Lefschetz cycles are on top of the ones we represent in this diagram. 
Using \autoref{lem:basickirby} we obtain the diagram in \autoref{fig:k=3a}.(a). Sliding the `$b$-handle' over the 1-framed unknot, that unknot becomes unlinked from the rest of the diagram, thereby splitting off a copy of~$\CP$.
Moreover, we can manipulate the remaining diagram into the shape of a Kirby diagram of a \blf by first creating an overcrossing for the~$-2$-framed~$2$-handle, so that its blackboard framing becomes $-1$ (see \autoref{fig:k=3a}.(c)) and then subtracting the $0$-framed $2$-handle representing $a$ from the $-2$-framed $2$-handle representing $b$ to obtain \autoref{fig:k=3a}.(d). 
The final effect on the fibration is the replacement of the singularities with vanishing cycles $b+3a=A^3(b)$ and $b$ by one with vanishing cycle $b-a=A\inv(b)$. 
In order to understand the effect on the divisor we compare the monodromies:
	\begin{align*}
	\tau_{b-a}^{-1}(\tau_b\circ \tau_{b+3a})&=A^{-1}B^{-1}ABA^3BA^{-3} = A^{-1}B^{-1}(ABA) A^2BA\\
	&=A^{-1}B^{-1}(BAB) A^2BA=A^{-1}(AB A) ABA\\
	&=A^{-1}(BA)^3= -A^{-1}.\\
\end{align*}
It follows that the co-orientability is reversed by the replacement.
\begin{figure}[h!!]
\begin{center}
\includegraphics[height=8cm]{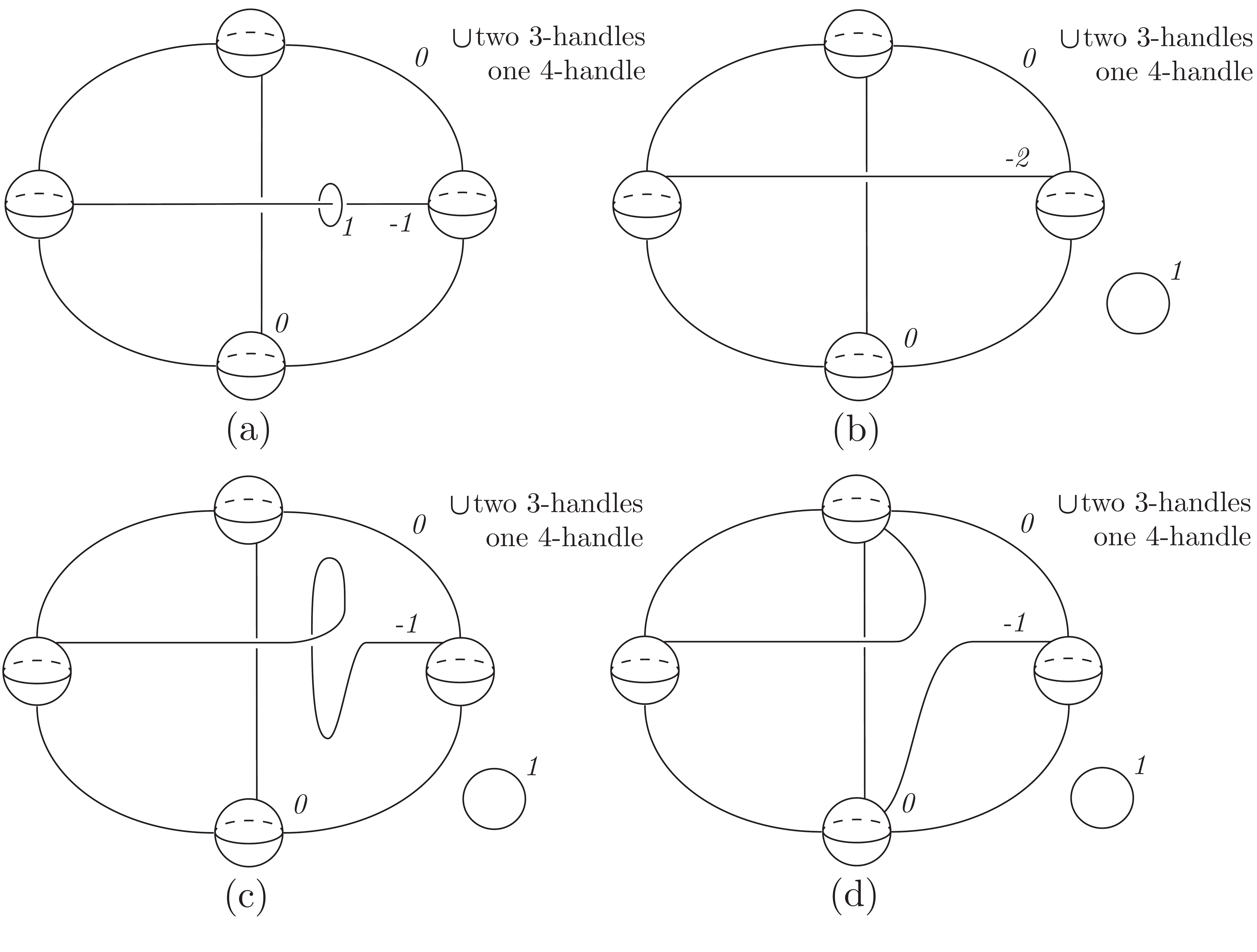} \caption{Case $k_1-k_2=3$. Manipulation of the Kirby diagram to split off a copy of $\C P^2$.} \label{fig:k=3a}
\end{center}
\end{figure}
\end{proof}

The case~$k=2$ in~\eqref{eq:minimal normal form} is a bit more complicated since it explicitly involves the third Lefschetz vanishing cycle.

\begin{proposition}\label{T:reduction 2}
Let $f\colon (X^4,\mc D^2)\ra (D^2,\del D^2)$ be a \blf over the disc with cycle system of the form 
	$\big(a;b+2a,b,b+na, b_4,\dots,b_\ell\big)$ with $n\in\Z$.
\begin{enumerate}
\item
If $n$ is even, then there is a \blf $f'\colon (X',\mc D')\ra (D^2,\del D^2)$ with two fewer Lefschetz singularities.
We have~$X=X'\#S^2\times S^2$ and the co-orientability of~$\mc D'$ is opposite to that of~$\mc D$.
A cycle system for~$f'$ is given by~$(a;b+na,b_4,\dots,b_\ell)$;
\item 
If $n$ is odd, the cycle system is equivalent to one of those covered by \autoref{T:reduction 1 and 3}.
\end{enumerate}
\end{proposition}
\begin{proof}
Before we start drawing Kirby diagrams, we show that we can gain some more control over~$n$, namely, we can can change it by arbitrary multiples of~$4$.
This step is not strictly necessary for our aims, but may be of independent interest as it leads towards a classification of Lefschetz fibrations over the disc which have signed powers of Dehn twists as monodromy.

\begin{lemma}\label{lem:k3}
The following holds:
$$\big(a;b+2a,b,b+na,b_4,\dots,b_\ell\big) \sim \big(a;b+2a,b,b+(n+4)a,A^{4}(b_4),\dots,A^{4}(b_\ell)\big).$$
\end{lemma}
\begin{proof}
As we saw in \autoref{T:two Lefschetz}, the monodromy around the pair of singularities with vanishing cycles~$b+2a$ and~$b$ is~$-A^{4}$, therefore, using Hurwitz moves and the fact that the vanishing cycles do not have a prefered orientation we have 
\begin{align*}
\big(a;b+2a,b,b+na,b_4,\dots,b_\ell\big) &\sim \big(a;-A^{4}(b+na),-A^{4}(b_4),\dots,-A^{4}(b_\ell),b+2a,b\big)\\
&\sim \big(a;A^{4}(b+na),A^{4}(b_4),\dots,A^{4}(b_\ell),b+2a,b\big)\\
&= \big(a;b+(n+4)a,A^{4}(b_4),\dots,A^{4}(b_\ell),b+2a,b\big)\\
&\sim \big(a;b+2a,b,b+(n+4)a,A^{4}(b_4),\dots,A^{4}(b_\ell)\big). \qedhere
\end{align*}
\end{proof}

With this lemma at hand, we can arrange that in Hayano's factorisation as in \autoref{T:factorization minimal case} the cycle system is $(a;b+2a,b,b+na,\dots)$, where $n=-3,-2,-1$ or $0$. It is worth looking at the four possibilities it yields. 
If $n = -1$, we note that
	\begin{equation*}
	(a;b+2a,b,b-a,\dots) \sim (a;b,b-a,\dots) \sim (a;b+a,b,\dots),
	\end{equation*}
which lands us back in case $(1)$ of~\autoref{T:reduction 1 and 3}. 
Similarly, if $n = -3$, then we have
	\begin{equation*}
	(a;b+2a,b,b-3a,\dots) \sim (a;b,b-3a,\dots) \sim (a;b+3a,b,\dots),
	\end{equation*}
which lands us in case $(2)$ of~\autoref{T:reduction 1 and 3}. 
What remains are the cases in which $n = 0$ or~$-2$. 
We argue that these cases are, in fact, Hurwitz equivalent.
A quick computation shows that $\tau_{b+2a}\inv(b)=-b-4a$ which is just~$b+4a$ with the opposite orientation.
Hence we have
	\begin{align*}
	(a;b+2a,b,b,\dots) &\sim (a;\tau_{b+2a}\inv b,b+2a,b,\dots)\\
				   &=(a;b+4a,b+2a,b,\dots)\\
				   &\sim (a;b+2a,b,b-2a,\dots).
	\end{align*}

Now we can deal with the case $n = 0$ by drawing the Kirby diagram for the fibration. 
In what follows we will work only with the handles corresponding to the boundary vanishing cycle and the first three Lefschetz singularities, so we will omit the remaining $2$-handles with the understanding that they remain unchanged and lay on top of the handles where the interesting part takes place. Using \autoref{lem:basickirby}, this simplified Kirby diagram is drawn in \autoref{fig:k=2b}.(a). Sliding one $2$-handle representing $b$ over the other we obtain the diagram in \autoref{fig:k=2b}.(b) and we can slide the $2$-handle representing $b$ over the $0$-framed $2$-handle to split off a copy of $S^2 \times S^2$ from the diagram. Finally we observe that after removal of the $S^2 \times S^2$-factor, the remaining part is the Kirby diagram for the fibration with the singular fibres corresponding to $b+2a$ and $b$ removed. Since the monodromy around these is $-a^4$, the sign of the monodromy map for this new fibration is opposite to that of the original one.
\begin{figure}[h!!]
\begin{center}
\includegraphics[height=4.7cm]{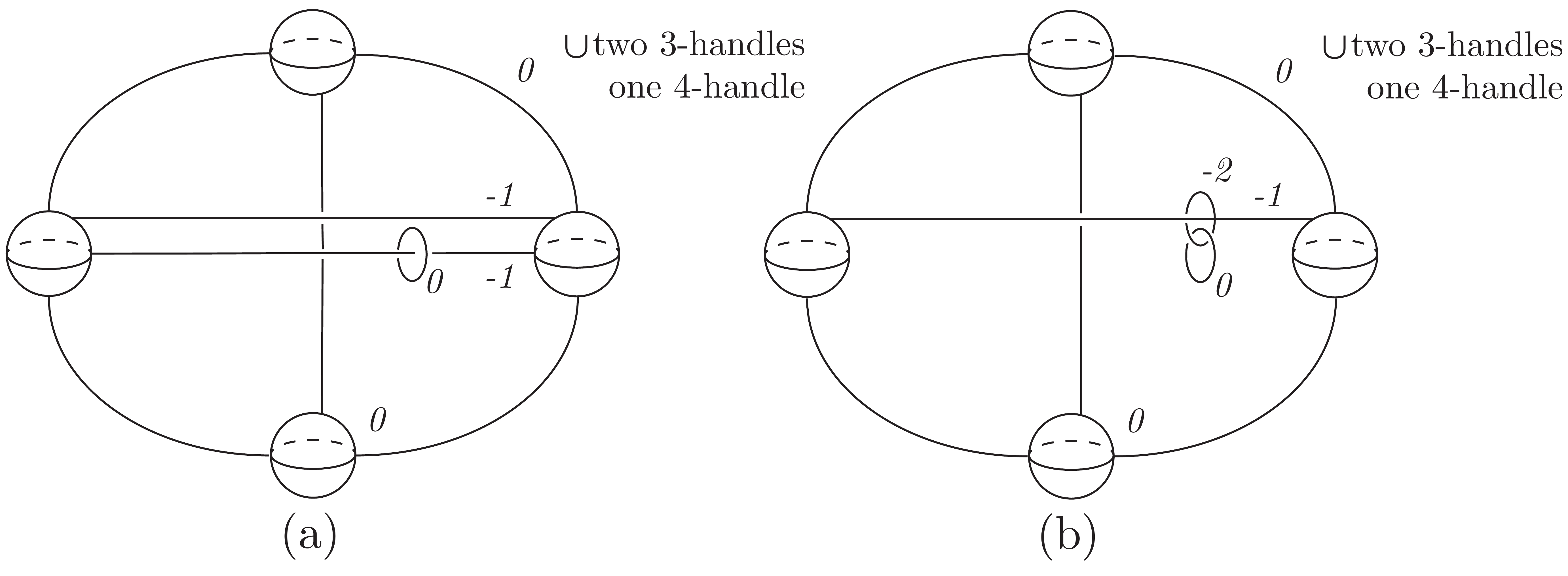} \caption{Case $k_2=-2$. Lefschetz fibration and corresponding Kirby diagram.} \label{fig:k=2b}
\end{center}
\end{figure}
\end{proof}

We now have the necessary tools to prove our main theorem:

\begin{theorem}\label{theo:blf over D2}
Let $f\colon (X^4,\mc D^2)\ra (D^2,\del D^2)$ be a relatively minimal boundary Lefschetz fibration.
Then $X$ is diffeomorphic to one of the following manifolds:
\begin{enumerate}
\item $S^1 \times S^3$;
\item $\# m S^2 \times S^2$, including $S^4$ for $m=0$;
\item $\# m \C P^2 \# n \overline{\C P^2}$ with $m> n \geq 0$.
\end{enumerate}
In all cases the generic fibre is nontrivial in $H_2(X\setminus \mathcal{D};\R)$. In case \emph{(1)}, $\mathcal{D}$ is co-orientable, while in cases \emph{(2)} and \emph{(3)}, $\mathcal{D}$ is co-orientable if and only if $m$ is odd.
\end{theorem}

\begin{proof}
Firstly, recall from \cite[Theorem 8.1]{2017arXiv170303798C} that the fibres of every \blf\ over the disc are homologically nontrivial on $X \backslash \mathcal{D}$ because $X \backslash \mathcal{D}$ is obtained from the trivial fibration by adding Lefschetz singularities. Topologically, each of these added singularities corresponds to the addition of a $2$-cell to $D^2 \times T^2$ which does not kill homology in degree $2$.

The theorem is true for fibrations with at most two Lefschetz singularities by 
	\autoref{T:no Lefschetz}, \autoref{T:one Lefschetz}, and \autoref{T:two Lefschetz}.
Finally, whenever there are three or more Lefschetz singularities, Hayano's factorisation theorem in the form of~\autoref{T:factorization minimal case} shows that we can apply either
	\autoref{T:reduction 1 and 3} or \autoref{T:reduction 2}
to pass to a \blf with fewer Lefschetz points, whilst spliting of a copy of either $\CPbar$, $\CP$, or~$S^2\times S^2$.
As for the effect on the divisor, observe that each time we split off or add a connected summand that contributes to $b_2^+$, there is a change in co-orientability. The base case, $S^4$, has a negative sign (see \autoref{T:two Lefschetz}), hence, if $f\colon (X^4,\mc{D}^2) \to (D^2,\del D^2)$ is a boundary Lefschetz fibration with co-orientable $\mathcal{D}$, the number $b_2^+(X)$ must be odd and vice versa.
\end{proof}

As a final step, we observe that all the replacements used in the reduction process can be reversed.
This allows us to produce \blf on all the manifolds listed in \autoref{theo:blf over D2}.

\begin{corollary}\label{T:reduction reversed}
Let $(a;b_1,\dots,b_\ell)$ be a cycle system.
\begin{enumerate}
\item
Passing to $(a;a,b_1,\dots,b_\ell)$ realizes a connected sum with~$\CPbar$.
The co-orientability of the divisor is preserved;
\item
If $\scp{a,b_1}=1$, then passing to $(a;b_1+4a,b_1+a,\dots,b_\ell)$ realizes a connected sum with~$\CP$.
The co-orientability of the divisor is reversed;
\item
If $\scp{a,b_1}=1$, then passing to $(a;b_1,b_1-2a,b_1\dots,b_\ell)$ realizes a connected sum with~$S^2\times S^2$.
The co-orientability of the divisor is reversed.
\end{enumerate}
\end{corollary}
\begin{proof}
This follows readily from \autoref{T:blow-down}, \autoref{T:reduction 1 and 3}, and \autoref{T:reduction 2}.
\end{proof}

\bibliographystyle{hyperamsplain-nodash}
\bibliography{references}
\end{document}